\documentclass[12pt, notitlepage]{amsart}
\usepackage{latexsym, amsfonts, amsmath, amssymb, amsthm, cite, tabls, paralist}

\newtheorem{theorem}{Theorem}[section]
\newtheorem{lemma}[theorem]{Lemma}
\newtheorem{proposition}[theorem]{Proposition}

\newtheorem*{nonumtheorem}{Theorem}
\newtheorem{corollary}[theorem]{Corollary}

\newtheorem{definition}[theorem]{Definition}
\newtheorem{example}[theorem]{Example}

\newcommand{\Z}{\mathbb{Z}}

\newcommand{\prob}{\mathbf{P}}

\newcommand{\C}{\mathbb{C}}

\newcommand{\cg}[1]{\Z/#1\Z}

\usepackage{graphicx}
\usepackage{mathrsfs}

\usepackage{varioref}
\labelformat{section}{Section~#1}
\labelformat{subsection}{Section~#1} \labelformat{table}{Table~#1}

\begin{document}

\title{Carries, group theory, and additive combinatorics}
\author{Persi Diaconis, Xuancheng Shao, Kannan Soundararajan}
\address{Department of Mathematics\\ Stanford University\\ 450 Serra Mall, Bldg. 380\\ Stanford, CA 94305-2125}
\email{fshao@stanford.edu}
\email{ksound{@}math.stanford.edu}
\thanks{The first author is supported in part by NSF grant DMS 08-04324.  The third author is supported in part by NSF grant DMS-1001068, and a Simons Investigator award from the Simons Foundation}
\maketitle

\section{Introduction}

 When numbers are added in the usual way {\em carries} occur along the
 route. These carries cause a mess and it is natural to seek ways to
 minimize them. This paper proves that {\em balanced arithmetic}
 minimizes the proportion of carries. It also positions carries as
 {\em cocycles} in group theory and shows that if coset
 representatives for a finite-index normal subgroup $H$ in a group
 $G$ can be chosen so that the proportion of carries is less than
 $2/9$, then there is a choice of coset representatives where no
 carries are needed (in other words, the extension {\em splits}).
 Finally, our paper makes the link between the problems above and
 the emerging field of additive combinatorics. Indeed the tools and
 techniques of this field are used in our proofs, and our examples
 provide an elementary introduction.

\subsection{Carries}
 \begin{example}
Table 1 shows a carries matrix for base $b=10$. Thus when $0$ is
added to one of the digits $0,1,\cdots,b-1$, no carries occur. When
$1$ is added, there is a carry of $b$ at $b-1$. There is a carry of
$b$ in position $i,j$ if and only if $i+j\geq b$.
 \end{example}

\begin{table}[htb]
\caption{Carries matrix for $b=10$. There is a carry of $b$ if and
only if $i+j\geq b$.}
\begin{small}\begin{equation*}\begin{array}{c|cccccccccc|}
&0&1&2&3&4&5&6&7&8&9\\\hline
0&0&0&0&0&0&0&0&0&0&0\\
1&0&0&0&0&0&0&0&0&0&b\\
2&0&0&0&0&0&0&0&0&b&b\\
3&0&0&0&0&0&0&0&b&b&b\\
4&0&0&0&0&0&0&b&b&b&b\\
5&0&0&0&0&0&b&b&b&b&b\\
6&0&0&0&0&b&b&b&b&b&b\\
7&0&0&0&b&b&b&b&b&b&b\\
8&0&0&b&b&b&b&b&b&b&b\\
9&0&b&b&b&b&b&b&b&b&b\\\hline
\end{array}\end{equation*}\end{small}
\label{tab1}
\end{table}

For an arbitrary base $b>1$ with digits $0$, $1$, $\ldots$, $b-1$, the corresponding matrix has
$\binom{b}{2}$ carries.  If the digits are
 chosen uniformly at random, the chance of a carry is $\binom{b}{2}/b^2 = \frac 12 - \frac{1}{2b}$.
 This is $45\%$ when $b=10$.

If $b\Z\subset\Z$ is the subgroup $\{0,\pm b,\pm 2b,\cdots\}$ and
coset representatives are chosen as $\{0,1,2,\cdots,b-1\}$, the
carries are cocycles \cite{isaksen}: $i+j=(i+j)_b+f(i,j)$ with $(i+j)_b$ the
sum modulo $b$ and $f(i,j)$ the `remainder'. Here $f(i,j)=0$ when
$i+j<b$ and $f(i,j)=b$ when $i+j\geq b$. It is natural to ask if
some other choice of coset representatives has fewer carries. The
answer is classically known.

\begin{example}
 For simplicity, take $b$ odd. The
balanced representatives $\{0,\pm 1,\cdots,\pm\tfrac{b-1}{2}\}$ lead
to about half as many carries. For example, when $b=5$, the carries
table for $5\Z\subset\Z$ is shown in \ref{tab2}.
\end{example}

\begin{table}[htb]
\caption{Carries matrix for $b=5$ with signed coset representatives $\{ 0, \pm 1, \pm 2\}$.  Here
$-i$ is coded as $\bar{i}$.}
\begin{small}\begin{equation*}\begin{array}{c|ccccc|}
&\bar2&\bar1&0&1&2\\\hline
\bar2&\bar{b}&\bar{b}&0&0&0\\
\bar1&\bar{b}&0&0&0&0\\
0&0&0&0&0&0\\
1&0&0&0&0&b\\
2&0&0&0&b&b\\\hline
\end{array}\end{equation*}\end{small}
\label{tab2}
\end{table}

For example $(-2)+(-2)=-5+1$ and $2+2=5-1$. The balanced
representatives lead to $6$ carries while the usual choice leads to
$\binom{5}{2}=10$. Signed digit representations have a long history
going back to Colson \cite{colson} and Cauchy \cite{cauchy}. A
careful history is in Cajori \cite{cajori} with Knuth \cite{knuth}
giving further details. The study of carries has links to probability \cite{holte,diaconis2009a} and various parts of algebra \cite{borodin}.

Can one do better?  Why do there have to be any carries?  What is
the best that can be done? These are problems in additive
combinatorics. If $X$ is a choice of coset representatives for $b\Z$
in $\Z$, we are asking  for connections between $X$ and its sumset $X+X$.

\subsection{Group theory}

These questions make sense for any group. For example, the matrix in
\ref{tab1} and \ref{tab2} also give the carries for the cyclic group
$\Z/b\Z\subset\Z/b^2\Z$ with coset representatives
$\{0,1,2,\cdots,b-1\}$ or $\{0,\pm 1,\cdots,\pm (b-1)/2\}$ and
everything interpreted modulo $b^2$. The proofs for $\Z$ do not
carry over to $\Z/b^2\Z$ since $i+j$ might collapse to a coset
representative modulo $b^2$.

Let us now formulate the carries problem precisely when $G$ is a group,
and $H$ a finite index normal subgroup.  Let $X\subset G$ be coset representatives
for $H$ in $G$.   Given two elements $x_1$ and $x_2$ in $X$, there
is a unique third element $x_{12}\in X$ such that $x_{12}^{-1}x_1x_2$ lies in the
subgroup $H$.  Note that if we multiply $x_1 h_1$ and $x_2 h_2$ the
answer is $x_1 x_2 (x_2^{-1} h_1 x_2 h_2) = x_{12}( x_{12}^{-1} x_1 x_2) (x_2^{-1}h_1x_2h_2)$.
In analogy with the usual addition, we view
$x_{12}^{-1} x_1 x_2$ as the {\sl carry} in performing this multiplication.
Thus carries are elements of the subgroup $H$, and a (non-trivial) carry  occurs
for $x,y\in X$ exactly when $x\cdot y$ is not in $X$.

If $X$ is a subgroup (so that necessarily $XH=HX=G$ and $H\cap X=\{1\}$), there are
no carries and the extension $H\subset G$ is said to \textit{split}. For
$\Z/b\Z\subset\Z/b^2\Z$, any choice of coset representatives has $b$
elements and $\Z/b\Z$ is the unique subgroup of $\Z/b^2\Z$ of order
$b$, so the extension fails to split.  Our main theorem shows that
if the extension $H \subset G$ is not split, then there must be
many carries.   To quantify this notion, let us define
\[
 C(X)=\frac{|\{x,y\in X:xy\in X\}|}{|X|^2}.
 \]





\begin{theorem}\label{thm:7/9}
Let $X$ be coset representatives for a normal, finite index subgroup
$H$ in a group $G$. If
\[ C(X)>7/9 \]
then there is a subgroup $K$ with $HK=G$, $H\cap K=\{1\}$.
\end{theorem}

From Theorem 12 on page 182 of \cite{DF} for example, one sees that the
structure of $G$ above may be described as the semi-direct product of the
normal subgroup $H$ and the group $K$.
Further, the constant $7/9$ is sharp as seen by taking $3\Z\subset\Z$ with
balanced coset representatives.

\subsection{Additive combinatorics}

The problems discussed above may be seen as part of additive
combinatorics. A basic question in this area asks how the size $|X\cdot X|$ depends on the
structure of $X$. If $X$ is a subgroup, then $|X\cdot X|=|X|$. For a random set, one may expect
$X\cdot X$ to have about $|X|^2$ elements.  What happens $X\cdot X$ contains unusually few
elements; for example what if $|X\cdot X|\leq 2|X|$?  The structure of such sets is
studied in additive combinatorics, which is a burgeoning
area of mathematics with applications in computer science
\cite{trevisan}, harmonic analysis \cite{laba}, number theory \cite{nathanson},
combinatorics and elsewhere. It has spanned a host of new techniques
(e.g. Szemer\'{e}di's regularity lemma \cite{szemeredi,szemeredi-survey}, higher Fourier
analysis \cite{gowers1,gowers2,higher-fourier}). It gives connections between formerly disparate
areas of mathematics (e.g. combinatorics, number theory and ergodic
theory). There are striking results, such as the Green-Tao theorem
that the primes contain arbitrarily long arithmetic progression
\cite{green-tao1,green-tao2}.

Our theorems offer a gentle introduction to this field in a natural
problem. The proof of the main theorem uses results on {\em
approximate homomorphisms} first studied by computer scientists for
property testing (the study of large systems from properties of
small samples). A second related result is given in
\ref{sec:59/60}; here the situation is more general than in Theorem \ref{thm:7/9} but the
conclusion is weaker. This uses an argument of Fournier, familiar in
  additive combinatorics, to show that any finite subset $X$ of a group $G$ is almost a
subgroup if $C(X)$ is large.

Our route to the discovery and proof of Theorem \ref{thm:7/9} has
some lessons. Our first results were limited to $p\Z/p^2\Z$ in
$\Z/p^2\Z$, and they were asymptotic: if $X$ is a set of coset
representatives then $C(X)\leq 3/4+\epsilon$ provided $p$ is a
sufficiently large prime (depending on $\epsilon$). Here the
dependence on $\epsilon$ is exponential. The argument uses
\textit{rectification} \cite{bilu-lev-ruzsa,green-ruzsa}, which
roughly speaking converts additively structured subsets of $\Z/p\Z$
to subsets of $\Z$. Later we found out that we could get rid of the
asymptotics, proving that $C(X)\leq (3p^2+1)/(4p^2)$ for any odd
prime $p$ using a theorem of Lev \cite{lev}. This was done
independently by Alon \cite{alon}. All of these arguments rely on
the primality of the base $p$.

\ref{sec:easy} gives a very easy proof of the optimality of balanced
coset representatives for $b\Z\subset\Z$. Theorem
\ref{thm:7/9} is proved in \ref{sec:7/9}. A different proof
(with $C(X)\geq 59/60$ implying splitting) appears in
\ref{sec:59/60}. In Table 2 above there are three types of carries:
$0,+b,-b$, while only $0$ and $+b$ appear with the usual choice of
digits. This is shown to characterize the usual digits in
\ref{sec:twocarries}. The final section presents some problems and
conjectures. We do not know the answer to some simple related
questions: how well can one do for $(b\Z)^2\subset\Z^2$?

\smallskip

{\bf Acknowledgments.}   We are grateful to Ben Green, Bob Guralnick and Marty Isaacs
for many valuable discussions.


\section{The easiest case: Minimality of balanced digits for $\Z$}\label{sec:easy}

For $b$ a positive integer, consider $b\Z\subset\Z$. Choose coset
representatives $\mathfrak{X}=\{0,x_1,x_2,\cdots,x_{b-1}\}$ in $\Z$.
There is a carry at $i,j$ if $x_i+x_j\notin\mathfrak{X}$. The
following proposition shows that any choice for $\mathfrak{X}$
results in at least $\lfloor b^2/4\rfloor$ carries. Balanced coset
representatives give this and so are best (in this sense). In fact,
the argument works for any set of $b$ real numbers.

\begin{proposition}\label{prop:Z}
Let $\mathfrak{X}=\{0,x_1,\cdots,x_{b-1}\}$ be distinct real
numbers. Then $\mathfrak{X}$ induces at least $\lfloor b^2/4\rfloor$ carries.
\end{proposition}

\begin{proof}
Let there be $c$ positive and $(b-1-c)$ negative elements in
$\mathfrak{X}$. Say $0<y_1<y_2<\cdots<y_c$ are the positives. Then,
adding $y_c$ results in at least $c$ carries. Adding $y_{c-1}$
results in at least $c-1$ carries. Continuing in this fashion,
adding $y_1$ results in at least $1$ carry. This forces at least
$c(c+1)/2$ carries. Similarly, the negative elements in
$\mathfrak{X}$ force at least $(b-1-c)(b-c)/2$ carries, thus obtaining
altogether
\[
\tfrac{1}{2}[c(c+1)+(b-1-c)(b-c)] = \frac{b^2-1}{4} + \Big( c- \frac{b-1}{2}\Big)^2
\]
carries.   This proves the Proposition.
\end{proof}

By examining the above proof, we may check that $\lfloor b^2/4\rfloor$ carries
are  attained only if $\mathfrak {X}$ is of the form $\{ xn:  \, - \lfloor b/2\rfloor < n\le \lfloor b/2\rfloor\}$
for some $x\neq 0$.  Thus, for $b\Z \subset \Z$ balanced coset representatives and
their dilates by any number $a$ relatively prime to $b$ are the only examples with $\lfloor b^2/4\rfloor$
carries.


In the other direction, it is easy to (foolishly) choose coset
representatives $\mathfrak{X}$ for $b\Z$ in $\Z$ such that every sum
results in a carry. For example choose $\{b,b+1,\cdots,2b-1\}$.


\section{The next case: Minimality of balanced digits for cyclic groups}\label{sec:cyclic}

This section studies the following problem: consider
$p(\Z/p^2\Z)$ as a subgroup of $\Z/p^2\Z$ for an odd prime $p$. The usual coset representatives are
$\{0,1,2,\dots,p-1\}$. Balanced coset representatives are $\{0,\pm
1,\dots,\pm (p-1)/2\}$. The carries matrices are the same
as for $p\Z\subset\Z$. The following proposition implies that balanced coset representatives again give the minimum number of carries.

\begin{proposition}\label{prop:Zp}
Let $p$ be an odd prime. Let $X\subset\Z/p^2\Z$ be coset representatives for the subgroup $p(\Z/p^2\Z)$ in $\Z/p^2\Z$. Then $X$ induces at least $(p^2-1)/4$ carries.
\end{proposition}

Proposition \ref{prop:Zp} is a consequence of the following result,
proved below.
\begin{proposition}\label{prop:sr}
Let $p$ be an odd prime. Let $A_1,A_2,A_3\subset\cg{p^2}$ be three
sets of coset representatives for $p (\cg{p^2}) \subset\cg{p^2}$. Then the
number of solutions to $a_1+a_2=a_3$ with $a_1\in A_1$, $a_2\in
A_2$, and $a_3\in A_3$ is at most $(3p^2+1)/4$.
\end{proposition}

The problem of counting the number of solutions to linear equations
in finite fields has been studied in \cite{lev}. The strategy there
is to use Pollard's theorem \cite{pollard}. Our situation is
slightly different in that we are working in $\cg{p^2}$, which is
not a finite field. However, we can still follow the argument in
\cite{lev}, making use of a version of Pollard's theorem for
composite modulus \cite{pollard}.
\begin{theorem}[Pollard]
Let $m$ be a positive integer. Let $A_1,A_2,\dots,A_k$ be subsets of
$\cg{m}$ and let $A_1',A_2',\dots,A_k'$ be another $k$ subsets of
$\cg{m}$ such that each $A_i'$ consists of consecutive elements and
has $|A_i'|=|A_i|$. Write
\[ S(A_1,A_2,\dots,A_k,r)=\sum_{x\in\cg{m}}\min(r,n(x,A_1,A_2,\dots,A_k)), \]
where $n(x,A_1,A_2,\dots,A_k)$ is the number of representations of
$x$ as $x=a_1+a_2+\dots+a_k$ ($a_i\in A_i$). Define
$S(A_1'A_2',\dots,A_k',r)$ similarly. Suppose that at least $k-1$ of
the sets $A_i$ have the property that
\[ (x-y,m)=1\text{ for }x,y\in A_i\text{ and }x\neq y. \]
Then
\[ S(A_1,A_2,\dots,A_k,r)\geq S(A_1',A_2',\dots,A_k',r). \]
\end{theorem}

To gain an appreciation of Pollard's theorem, consider the special case $m=p$ a prime, $k=2$ and $r=1$.
When $m$ is prime, the hypothesis in Pollard's theorem is automatically satisfied.  Now $S(A_1,A_2,1)$
counts the number of elements in the sumset $A_1+A_2$, and Pollard's theorem gives that this
cardinality is smallest when $A_1$ and $A_2$ are intervals.  It thus follows that $|A_1+A_2|
\ge \min( p, |A_1|+|A_2|-1)$, which is a fundamental result on set addition known as
the Cauchy-Davenport theorem (a result proved by Cauchy in 1813, and rediscovered by Davenport
in 1935).  Thus Pollard's theorem may be viewed as a
generalization of the Cauchy-Davenport  result.  There has also been extensive
work on extending the Cauchy-Davenport  theorem, leading up to Kemperman's very general
theorem \cite{Kemp}; see Serra \cite{Ser} for a recent survey.

For the general case of Pollard's theorem, consider for each natural number $\ell$
the set $S_\ell$ of those elements in ${\Bbb Z}/m{\Bbb Z}$ which can be expressed as $a_1+\ldots+a_k$
in at least $\ell$ ways.  Then $S(A_1,A_2,\ldots, A_k,r)$ equals the
sum of the cardinalities of $S_\ell$ for all $1\le \ell \le r$.

\begin{corollary} With notation as in Pollard's theorem
$$
\max_{x} n(x,A_1,\ldots,A_k) \le \max_{x} n(x,A_1^{\prime},\ldots,A_k^{\prime}).
$$
\end{corollary}
\begin{proof}  Suppose the corollary does not hold, and take $r = \max_x n(x, A_1^{\prime},\ldots,A_k^{\prime})$ in
Pollard's theorem.  Note that
$$
S(A_1^{\prime},\ldots, A_k^{\prime},r) = \sum_{x} n(x,A_1^{\prime}, \ldots, A_k^{\prime}) = |A_1^{\prime}|\cdots |A_k^{\prime}|.
$$
On the other hand, since by assumption $r< n(x,A_1,\ldots, A_k)$ for some $x$,
$$
S(A_1,\ldots, A_k,r) = \sum_x \max(r, n(x,A_1,\ldots,A_k))
< \sum_x n(x,A_1,\ldots, A_k) = |A_1| \cdots |A_k|.
$$
But this contradicts Pollard's theorem, proving the Corollary.
\end{proof}
\begin{proof}[Proof of Proposition \ref{prop:sr}]
Since $A_1$, $A_2$ and $A_3$ consist of coset representatives for $p (\cg{p^2})$ in $\cg{p^2}$,
the hypothesis in Pollard's theorem is satisfied.  Now take $A_1^{\prime} = A_2^{\prime}=A_3^{\prime} =I$
where $I$ is the interval of length $p$ centered around the origin.  A simple calculation gives that
$$
\max_x n(x, I, I, I) = n(0,I, I, I) = \frac{3p^2+1}{4}.
$$
By Corollary 3.4 it follows that $n(0,A_1,A_2,-A_3)$ is at most $(3p^2+1)/4$.  Since $n(0,A_1,A_2,-A_3)$
precisely counts the number of solutions to $a_1+a_2=a_3$, the Proposition follows.
\end{proof}






\section{Carries and Approximate Homomorphisms}\label{sec:7/9}

This section proves Theorem \ref{thm:7/9} and gives an introduction
to computer scientists' use of approximate homomorphisms in
cryptography and for verifying program correctness.

\begin{definition}[Approximate homomorphisms]
Let $G_1,G_2$ be arbitrary groups with $G_1$ finite. Let
$\epsilon>0$. A function $f:G_1\rightarrow G_2$ is an
$\epsilon$-homomorphism if, picking $g,g'$ independently and
uniformly in $G_1$,
\[ \prob_{g,g'\in G_1}\{f(g)f(g')=f(gg')\}\geq\epsilon. \]
\end{definition}

Checking if a given program or black box is a homomorphism occurs in
cryptography (e.g. checking a random number generator) and in
program checking (e.g. does this matrix multiplication package
really work). Here is a brief description.

\textit{Cryptography.} Despite recent advances, many cryptography
schemes in active use still proceed by taking a message, given as a string of
letters in a finite field $x_1x_2\cdots x_N$, adding \textit{noise}
$\epsilon_1\epsilon_2\cdots\epsilon_N$ to each coordinate, and
sending $x_i+\epsilon_i=y_i$. A receiver in possession of the recipe
for the noise $\epsilon_i$ decodes via $y_i-\epsilon_i=x_i$. The
noise is usually generated by a pseudorandom generator. For example,
if the field is $\Z/p\Z$, the generator might be
$\epsilon_{i+1}=a\epsilon_i+b\pmod p$. Another scheme has the field
$\Z/2\Z$, breaks the message into blocks: $X_1=(x_1\cdots x_{256})$,
$X_2=(x_{257}\cdots x_{512}),\cdots$, and adds vectors of noise
$\tilde{\epsilon}_1,\tilde{\epsilon}_2,\cdots$. These
$\tilde{\epsilon}_i$ are often generated by a simple scheme such as
$\tilde{\epsilon}_{i+1}=A\tilde{\epsilon}_i$ with $A$ a fixed
$256\times 256$ matrix. Someone interested in checking this
generator has to determine $(a,b)$ (or $A$) and the initial seed. A
first task is to decide if such a linear scheme is in use. This
entails testing if the output is a homomorphism! For background and
a fascinating success story in online poker, see \cite{poker}.

\textit{Program checking.} A host of computer scientists have
developed a sophisticated suite of programs for testing if programs
designed to do standard numerical tasks are doing their job. A
readable entry to this literature is \cite{blr} and their
references. As an example, consider a program $P$ to multiply two
$n\times n$ matrices $A,B$ with elements in a finite field. Given
$A,B$, the program outputs $P(A,B)$. A complete test is out of the
question. A test which proves correctness with high probability is
suggested in \cite{blr}. Given $A,B$, form random uniform matrices
$A_1,B_1$. Set $A_2=A-A_1$, $B_2=B-B_1$, and
$C=P(A_1,B_1)+P(A_1,B_2)+P(A_2,B_1)+P(A_2,B_2)$. If the program is
working then $C=P(A,B)$ by simple algebra. The tools of approximate
homomorphisms are used to show this test (amplified by repetitions)
gives an efficient check which works with arbitrarily high
probability. While the examples above involve homomorphisms between
abelian groups $(\Z/p\Z)^{n^2}$, the theorists developed their tools
for general groups. One of their theorems turns out to be just what
we need to prove Theorem \ref{thm:7/9}.

The following theorem, due to Ben-Or, Coppersmith, Luby, and
Rubinfeld \cite{approx_hom}, says that for $\epsilon>7/9$, an
$\epsilon$-approximate homomorphism must coincide with a genuine
homomorphism on a large subset of $G_1$.

\begin{theorem}[Structure theorem for approximate homomorphisms]\label{thm:approx_hom}
Let $G_1,G_2$ be arbitrary groups with $G_1$ finite. Suppose that
$f:G_1\rightarrow G_2$ is an $\epsilon$-approximate homomorphism for
some $\epsilon>7/9$. Then there is a genuine homomorphism
$\phi:G_1\rightarrow G_2$ such that $\prob_{g\in
G_1}(f(g)\neq\phi(g))\leq\tau$, where $\tau=\tau(\epsilon)$ is the smaller root of
the equation $3x-6x^2=1-\epsilon$.
\end{theorem}

Note that $\tau(\epsilon)$ equals $(3-\sqrt{24\epsilon -15})/12$, and so $\tau(\epsilon) < (3-\sqrt{11/3})/12 = 0.0904\ldots$
when  $\epsilon >7/9$.    Both the range $\epsilon >7/9$ and the parameter $\tau(\epsilon)$
are sharp.   The
genuine homomorphism $\phi$ in the statement is constructed by
taking $\phi(g)$ to be the most frequent value of $f(gg')f(g')^{-1}$ over
all $g'\in G_1$. Under the stated assumptions, it can be shown that
this most frequent value is well-defined, the resulting map $\phi$ is a
genuine homomorphism, and it well approximates $f$.

\begin{proof}[Proof of Theorem \ref{thm:7/9}]   Since $H$ is a normal subgroup, the quotient $G/H$ forms
a group.  Consider now the map $f: G/H \to G$ that sends a coset $gH$ to its unique coset representative in $X$.
Given two cosets (along with their representatives in $X$), $gH = xH$ and $g^{\prime} H = x^{\prime} H$
note that $f(gH )f(g^{\prime}H) = f(gg^{\prime}H)$ if and only if $xx^{\prime}$ belongs to $X$.  In other words,
$f$ is a $C(X)$-approximate homomorphism.

Since $C(X)>7/9$ by hypothesis, Theorem \ref{thm:approx_hom} implies that
there is a genuine homomorphism $\phi: G/H \to G$ such that $f(gH) =\phi(gH)$ for
all but at most $\tau |G/H| < \frac 1{10} |G/H|$ cosets.   Let $K$ denote the image of the homomorphism $\phi$.
Thus $K$ is a subgroup of $G$ with $|K\cap X| \ge (1-\tau) |X| > \frac{9}{10}|X|=\frac{9}{10} |G/H|$.   By the
first isomorphism theorem $K$ is isomorphic to $(G/H)/\text{ker}(\phi)$
and therefore  the kernel of $\phi$ is trivial, and $|K|=|G/H|$.  If
$K$ contains an element $1\neq \ell \in H$, then for each $k\in K$ at most one
of $k$ or $k\ell$ can be in $X$; this would mean that $|K\cap X|\le |K|/2$ contradicting
our lower bound for $|K\cap X|$.  Thus $K\cap H=\{1\}$, and distinct elements of $K$ belong
to distinct cosets of $H$.  Therefore $K$ consists of  a complete set of
coset representatives for $H$ in $G$,
and we have $G=HK$, as desired.
\end{proof}

\section{An argument of Fournier}\label{sec:59/60}

In this section we study a problem that is a little more general than the
carries question.   Let $A$ be a finite set in a group $G$, and set
(in analogy with our earlier definition)
\[
C(A)=\frac{|\{a_1,a_2\in A:a_1a_2\in A\}|}{|A|^2}.
\]
The following result, which is established following an argument of Fournier,
 shows that if $C(A)$ is close to $1$, then $A$ is almost a subgroup.

\begin{theorem}\label{Fournier}  For a finite set $A$ in a group $G$, if $C(A) \ge 1-\delta$
for some $\delta \le 1/60$, then there exists a subgroup $K$ of $G$
such that
$$
|K| \le 10|A|/9 , \qquad \text{and} \qquad |A\cap K| \ge (1-5\delta)
|A|.
$$
 \end{theorem}

Let $A$ be any subset of $G$, and let $\epsilon$ be a real number in
$[0,1]$.  Define
$$
\text{Sym}_{1-\epsilon}(A) = \{ x \in G:  |A \cap Ax | \ge
(1-\epsilon)|A|\}.
$$
Since $|A\cap Ax| = |Ax^{-1} \cap A|$ the set
$\text{Sym}_{1-\epsilon}(A)$ is symmetric (that is, closed under
inverses). The following  monotonicity condition is clear:
$$
\text{Sym}_{1-\epsilon_1}(A) \subset \text{Sym}_{1-\epsilon_2}(A) \
\ \text{if } \ \ \epsilon_1 \le \epsilon_2.
$$
Observe further that if $x_1\in \text{Sym}_{1-\epsilon_1}(A)$ and
$x_2 \in \text{Sym}_{1-\epsilon_2}(A)$ then $x_1 x_2 $ lies in
$\text{Sym}_{1-\epsilon_1-\epsilon_2}(A)$.  To see this, note that
\begin{align*}
\#\{a\in A: ax_1x_2 \notin A\} &\le \#\{a \in A: \ ax_1 \notin A\} +
\#\{ a\in A: \ ax_1\in A, \ \ ax_1x_2\notin A\} \\
&\le \epsilon_1 |A| + \# \{b \in A: \ bx_2 \notin A\} \le
(\epsilon_1+\epsilon_2)|A|.
\end{align*}
The identity
$$
\sum_{x\in G} |A\cap Ax| =\sum_{x\in G} \sum_{\substack{ {a_1, a_2
\in A} \\ {a_1 =a_2 x}}} 1 = \sum_{a_1, a_2 \in A} \sum_{x=a_2^{-1}
a_1} 1  = |A|^2
$$
shows that
\begin{equation}
\label{Fournier1} |\text{Sym}_{1-\epsilon}(A)| \le |A|/(1-\epsilon).
\end{equation}

\begin{lemma}
\label{Fournier2}  Let $A$ be a subset of $G$ with $C(A) \ge
1-\delta$.  Then for any $\epsilon>\delta$ we have
$$
(1-\delta/\epsilon) |A| \le |A\cap \text{Sym}_{1-\epsilon}(A)|.
$$
\end{lemma}
\begin{proof}   Note that
$$
C(A) |A|^2 = \# \{a_1 a_2 = a_3\} = \sum_{a_2\in A} |A\cap Aa_2|.
$$
Now $|A\cap Aa_2|\le |A|$ for all $a_2\in A$, and $|A\cap Aa_2| \le
(1-\epsilon)|A|$ for $a_2$ lying in $A$ but not in
$\text{Sym}_{1-\epsilon}(A)$. Thus
$$
(1-\delta)|A|^2 \le |A| |A\cap \text{Sym}_{1-\epsilon}(A)| +
(1-\epsilon)|A| (|A|-|A\cap \text{Sym}_{1-\epsilon}(A)|),
$$
and the lemma follows upon rearranging.
\end{proof}

\begin{proof}[Proof of Theorem \ref{Fournier}]  With $\eta = 1/20$ we shall show that
$\text{Sym}_{1-2\eta}(A)$ equals $\text{Sym}_{1-4\eta}(A)$. Then
$\text{Sym}_{1-2\eta}(A) \times \text{Sym}_{1-2\eta}(A) \subset
\text{Sym}_{1-4\eta}(A) = \text{Sym}_{1-2\eta}(A)$, and it follows
that $\text{Sym}_{1-2\eta}(A) = \text{Sym}_{1-4\eta}(A)$ is a group.
This is the group $K$ of the Theorem. By \eqref{Fournier1} it
satisfies $|K|\le 10|A|/9$, and by Lemma \ref{Fournier2} we have
$|A\cap K| \ge (1-5\delta) |A|$; thus $K$ has the properties claimed
in the Theorem.

Since $\text{Sym}_{1-2\eta}(A) \subset \text{Sym}_{1-4\eta}(A)$, it
remains only to show the reverse inclusion. Consider any $x\in
\text{Sym}_{1-4\eta}(A)$.   The sets $\text{Sym}_{1-\eta}(A)$ and
$x\text{Sym}_{1-\eta}(A)$  both have cardinality at least
$|A|(1-\delta/\eta)$ by Lemma \ref{Fournier2}, and both are
contained in the set $\text{Sym}_{1-5\eta}(A)$ of cardinality at
most $|A|/(1-5\eta)$ by \eqref{Fournier1}.  Since $\delta <1/60$, we
deduce that $\text{Sym}_{1-\eta}(A)$ and $x\text{Sym}_{1-\eta}(A)$
have a non-empty intersection, and therefore $x$ may be written as
the product of two elements from $\text{Sym}_{1-\eta}(A)$.  Hence
$x$ must lie in $\text{Sym}_{1-2\eta}(A)$, completing the proof.
\end{proof}

\section{Characterizing the traditional choice of
digits}\label{sec:twocarries}

This section returns to the original setting of the cyclic groups
$p(\Z/p^2\Z) \subset\Z/p^2\Z$. The usual choice of coset
representatives $\{0,1,2,\cdots,p-1\}$ results in two types of
carries $\{0,p\}$ (Table 1). Balanced coset representatives (Table
2) need three types of carries $\{0,p,-p\}$. Random coset
representatives almost surely need all $p$ carries. The results
below show that two types of carries characterize the usual choice
of coset representatives. They use some basic tools of additive
combinatorics due to Freiman and make for a nice introduction to
these tools in a natural problem. At present the argument relies
on $p$ being prime, and it would be interesting to extend it to other groups.

\begin{theorem}\label{prop:char} Let $p$ be a prime, and let $A\subset\cg{p^2}$
be a set of coset representatives for $p(\cg{p^2})\subset\cg{p^2}$.
Suppose that the carries matrix associated to $A$ contains only two
distinct entries. Then there exist $c\in (\cg{p^2})^{\times}$ and
$d\in p(\cg{p^2})$ such that after dilating $A$ by $c$ and
translating by $d$ we have either $cA+d=\{0,1,\dots,p-1\}$ or
$cA+d=\{1,2,\cdots,p\}$.
\end{theorem}

 If the carries matrix for $A$ contains only two distinct entries then
 the sumset $A+A$ is contained in two translates of the set $A$ and thus $|A+A|\le 2|A|$.
Pollard's theorem tells us that $|A+A|\ge 2|A|-1$ (this is essentially the
Cauchy-Davenport theorem, as discussed in Section 3), and so our situation
is very close to the minimal possible doubling of a set.  Note that a typical random
set $A$ might be expected to have sumset $A+A$ as large as $|A|^2$ in size,
and one would expect sets with small doubling to be very structured and far from random.
This is the content of a celebrated theorem of Freiman, and we give a sample such
result in the case of subsets of the integers.

\begin{nonumtheorem}[Freiman's $3k-3$ theorem]\label{thm:fz}
Let $A\subset\Z$ with $|A|=k\geq 3$. If  $|A+A|=2k-1+b\leq 3k-4$
then $A$ is a subset of an arithmetic progression of length $k+b$.
\end{nonumtheorem}

Freiman's $3k-3$-theorem does not directly apply in our situation, since
we are dealing with a subset of ${\Bbb Z}/{p^2 {\Bbb Z}}$ rather than
a subset of ${\Bbb Z}$.  The problem is that the congruence $a+b\equiv c+d \pmod{p^2}$
does not necessarily mean that $a+b=c+d$ as an equation in the
integers.  Thus a sumset in ${\Bbb Z}/p^2 {\Bbb Z}$ could look very
different from a sumset in ${\Bbb Z}$.   However, if we could choose
representatives for the residue classes of $A \subset \cg{p^2}$ to lie always in
the interval $(-p^2/4,p^2/4]$ then the congruence $a+b\equiv c+d\pmod{p^2}$
is indeed equivalent to the equation $a+b=c+d$.   If this can be done, then we may
as well view $A$ as a subset of the integers and results such as Freiman's $3k-3$-theorem
would become applicable.   This is a case of a very useful notion of Freiman which
identifies when two subsets of different groups behave additively in a similar way.

\begin{definition}[Freiman isomorphism]  Let $A\subset G$ and $B\subset H$ be
two subsets of the abelian groups $G$ and $H$.  We say that $A$ and $B$ are
Freiman isomorphic if there is a bijection $\phi: A \to B$ such that the relation
$x+y=z+w$ holds with $x$, $y$, $z$, $w$ in $A$ if and only if
the relation $\phi(x)+\phi(y)=\phi(z)+\phi(w)$ holds in the group $H$.
\end{definition}

Note that if $A$ and $B$ are Freiman isomorphic then $|A+A|= |B+B|$.
Returning to our problem, we would like to show that our set $A\subset \cg{p^2}$
is Freiman isomorphic to a subset of the integers, and then
apply Freiman's $3k-3$ theorem.   This follows a strategy pioneered by Freiman himself,
who showed that  small subsets of ${\Bbb Z}/p{\Bbb Z}$ with
small doubling are isomorphic to subsets of the integers (also called {\sl rectifiable})  leading to the
following theorem (see Section 2.8 of Nathanson \cite{nathanson}).

\begin{nonumtheorem}[Freiman's $2.4$ theorem]\label{thm:fzp}
Set $c=1/35$ and $\alpha=2.4$. Let $A\subset\cg{p}$ with $|A|=k\leq
cp$. If $ |A+A|=2k-1+b\leq\alpha k-3$  then $A$ is contained in an
arithmetic progression in $\cg{p}$ of length $k+b$.
\end{nonumtheorem}

More recently Bilu, Lev and Ruzsa \cite{bilu-lev-ruzsa} and Green and Ruzsa \cite{green-ruzsa}
have shown how any small subset of $\cg{p}$ with small doubling
may be rectified.   By adapting these arguments to our setting of $\cg{p^2}$ we shall
establish the following Proposition.

\begin{proposition} Let $A\subset \cg{p^2}$ be a set of coset representatives for $p(\cg{p^2}) \subset \cg{p^2}$ and
suppose that $|A+A|\le 2|A|$.  Then there exists a dilation $c\in (\cg{p^2})^{\times}$ and a translation $d\in \cg{p^2}$ such
that $cA+d$ lies in $(-p^2/4,p^2/4]$.  Thus $A$ is Freiman isomorphic to a subset of the integers.
\end{proposition}

 Assuming this Proposition, let us now prove Theorem 6.1.

 \begin{proof}[Proof of Theorem 6.1 assuming Proposition 6.3] Let $A\subset\Z/p^2\Z$ be a set of coset
representatives with only two distinct carries, so that $|A+A|\leq
2|A|$.   By Proposition 6.3 we may dilate $A$ by some $c \in (\cg{p^2})^{\times}$
and obtain a set contained in $(d-p^2/4,d+p^2/4]$ for some $d \in \cg{p^2}$.
This means that $A$ is Freiman isomorphic to a subset of the integers, and applying Freiman's $3k-3$-theorem
we see that $A$ must lie in an arithmetic progression of length at most $|A+A|-|A|+1 \le (p+1)$.

  After a dilation if necessary, we may assume that $A$ lies in an
interval of length $p+1$, missing exactly one element from this
interval. Since $A$ consists of coset representatives, the missing
element must be one of the endpoints of the interval, so that
$A$ consists of consecutive elements; say $A = \{u, u+1, \ldots, u+p-1\}$ for some $u$.
 It remains to show that $u\equiv 0,1\pmod p$.

 To see this, if $u\equiv i\pmod p$ for
some $2\leq i\leq p-1$, then the following examples show that
there must be three types of carries:
\[ (u+p-1)+(u+p-1)-(u+i-2)=2p+u-i; \ \ u+u-(u+i)=u-i;
\]
and
\[
(u+p-1)+u-(u+i-1)=p+u-i. \]
This completes our proof.
\end{proof}

Now we turn to the proof of Proposition 6.3, whose argument involves
two parts.  First we establish a combinatorial result which shows that
if a substantial part of $A$ can be translated and dilated into the interval
$(-p^2/4,p^2/4]$ then all of $A$ can be.   This result holds for all cyclic
groups $\cg{m}$.  Second we use some simple Fourier analysis
to show that a large part of $A$ can be translated and dilated
into $(-p^2/4,p^2/4]$ so that our first argument may be used.  This
argument requires that we are working in $\cg{p^2}$.

\subsection{From a large subset to the entire set}

\begin{proposition}\label{prop:rect}
Let $m$ be a positive integer, and let $A$ be a subset
of $\cg{m}$ such that if $x \neq y\in A$ then $(x-y,m)=1$.  Of all the sets $cA+d$
(with $c\in (\cg{m})^*$ and $d\in \cg{m}$ let $\ell$ denote the maximum intersection of
such a set with $(-m/4,m/4]$.  Suppose that $\ell < |A|$.  Then either
$\ell < (|2A|+4)/3$ or $m\le 6(|2A|-\ell)$.
\end{proposition}
\begin{proof}  Let us suppose that $A$ has been already translated and dilated to have
maximum intersection with $(-m/4,m/4]$, and let $A_0$ denote this
intersection. Thus $|A_0| =\ell <|A|$ by assumption, and $A_0$ is
Freiman isomorphic to a subset of the integers.

Now $2A_0 \subset 2A$, and write $|2A| = 2\ell -1 + b$.  We may assume that $b\le \ell -3$,
else the first alternative in the Proposition holds.  Since $|2A_0| \le 2\ell -1 +b$, by Freiman's $(3k-3)$-theorem we see that $A_0$ is contained in an arithmetic progression of
size $\ell +b$.  Since the elements of $A$ (and hence $A_0$) satisfy that $(x-y,m)=1$, the
common difference of this arithmetic progression must be coprime to $m$.  Therefore
by translating and dilating (using dilations coprime to $m$) we may assume that $A_0$
is contained inside $(-(\ell+b)/2,(\ell+b)/2]$.

Since $\ell <|A|$, there must be an element $a \in A$ such that when reduced $\pmod m$,
$a$ lies either in $(m/4,m/2]$ or $(-m/2,-m/4]$.  Now the set $A_0+A_0$ has at least $2\ell -1$ elements, and all of these lie in
$(-\ell-b,\ell+b]$, and the set $A_0 +\{a\}$
has $\ell$ elements all lying in either $(m/4-(\ell+b)/2,m/2+(\ell+b)/2]$ or
$(-m/2-(\ell+b)/2,-m/4+(\ell+b)/2]$.  If the second alternative of the proposition doesn't
hold, then the sets $A_0+A_0$ and $A_0+\{a\}$ have at most one element in common,
and thus give at least $3\ell -2$ elements in $2A$ which
is a contradiction.
\end{proof}

 \subsection{Obtaining concentration near the origin}  Now we carry out the second part of
 the argument showing that a large part of $A$ can be put inside $(-p^2/4,p^2/4]$.

 \begin{proposition}[Concentration near the origin]\label{cor:rough}
Let $A\subset\Z/p^2\Z$ be as in the statement of Proposition 6.3. Then there exist $c\in (\Z/p^2\Z)^{\times}$ and $d\in\Z/p^2\Z$ such
that after dilating $A$ by $c$ and translating by $d$, we have
$$
|(cA+d)\cap (-p^2/4,p^2/4]| \ge \frac p2 \Big( 1 + \Big(\frac{p-2}{2(p-1)}\Big)^{\frac 12}\Big).
$$
\end{proposition}

   This
 uses a little Fourier analysis:  For $A\subset\Z/m\Z$ the Fourier
coefficients are defined by the formula
\[ \hat{A}(r)=\sum_{a\in A}e^{2\pi ira/m} \]
for $r\in\Z/m\Z$.

\begin{lemma}[Obtaining a large Fourier
coefficient]\label{lem:largeFC} Let  $m$ be a positive integer, and
let $A\subset \cg{m}$ be a subset. Write  $|A|=\alpha_1m$ and
$|A+A|=\alpha_2m$. Then
$$
\max_{r\neq 0} |{\hat A}(r)| \ge |A| \Big( \frac{\alpha_1
(1-\alpha_{2})}{\alpha_{2}(1-\alpha_1)}\Big)^{1/2}.
$$
\end{lemma}

\begin{proof}  Write $S=A+A$. Note that
$$
\alpha_1^2 m^2 = |A|^{2} = \sum_{\substack{a_1,a_2\in A\\
a_1+a_2\in S}}1.
$$
Using Parseval's identity this equals
$$
\frac 1m \sum_k {\hat A}(k)^{2} \widehat{S}(-k)= \alpha_1^{2}
\alpha_{2} m^{2} + \frac 1m \sum_{k\neq 0} {\hat A}(k)^{2}
\widehat{S}(-k).
$$
Thus
\begin{align*}
\alpha_1^2 (1-\alpha_2) m^2 = \frac 1m \Big| \sum_{k\neq 0} {\hat
A}(k)^2 \widehat{S}(-k)\Big| \le \Big(\max_{k\neq 0}|{\hat
A}(k)|\Big) \frac{1}{m} \sum_{k\neq 0} |{\hat A}(k)|
|\widehat{S}(-k)|.
\end{align*}
By Cauchy's inequality and Parseval's identity
\begin{align*}
\frac 1m \sum_{k\neq 0} |{\hat A}(k)| |\widehat{S}(-k)| &\le
\Big(\frac 1m \sum_{k\neq 0} |{\hat A}(k)|^2 \Big)^{\frac 12}
\Big(\frac 1m \sum_{k\neq 0} |\widehat{S}(-k)|^2 \Big)^{\frac 12}
\\
&= (\alpha_1(1-\alpha_1))^{\frac 12} (\alpha_2 (1-\alpha_2))^{\frac
12}m,
\end{align*}
and the Lemma follows with a little rearranging.
\end{proof}

We also require the following combinatorial result of Lev\cite{Lev2} (see also Theorem 2.9 of
Nathanson \cite{nathanson}).

\begin{lemma} \label{lem:nathanson}
Let $z_1,\dots,z_m\in\C$ be points on the unit circle. If
\[ |z_1+\dots+z_m|>2n-m+2(m-n)\cos\left(\phi/2\right),\]
then there exists an arc on the unit circle of length $\phi$
containing more than $n$ points.  In particular some arc of length
$\pi$ contains at least $\frac{1}{2} (m+ |z_1+\ldots +z_m|)$ points.
\end{lemma}

\begin{proof}[Proof of Proposition 6.5]  Apply Lemma \ref{lem:largeFC} with $m=p^2$, $\alpha_1=1/p$, and
$\alpha_2\le 2/p$, to get
\[
\max_{r\neq 0}  |\hat{A}(r)|\geq p\left(\frac{p-2}{2(p-1)}\right)^{1/2}.  \]
Since $A$ consists of coset representatives for $p(\cg{p^2}) \subset \cg{p^2}$ it follows that ${\hat A}(r)=0$
for those $r$ that are multiples of $p$ but not of $p^2$.   Thus the maximal non-zero
Fourier coefficient produced above is coprime to $p$.   Thus after dilating the original $A$ by $r$ if needed, we may
assume that the maximal Fourier coefficient is attained at $r=1$.

Now apply Lemma \ref{lem:nathanson} with the $p$ points $e^{2\pi ia/p^2}$ for $a\in A$. We conclude that
some arc of length $\pi$ contains at least $\frac 12 (p+|{\hat A}(1)|)$ points $e^{2\pi i a/p^2}$,
which is the Proposition.
 \end{proof}

\begin{proof}[Proof of Proposition 6.3]  When $p=2$ we may easily translate and dilate $A$ to equal $\{0,1\}$.
When $p=3$ or $5$ Proposition 6.5 already shows that $A$ may be translated and dilated
to lie inside $(-p^2/4,p^2/4]$.   For $p$ at least $7$,  a small calculation shows that Propositions 6.4 and 6.5 may be
combined to give the conclusion of Proposition 6.3.
 \end{proof}


\section{Open Problems}

For $X$ coset representatives for a normal, finite index subgroup
$H$ in an arbitrary group $G$, we have shown that either $G$ is a
semi direct product of $H$ and another subgroup $K$,  or $C(X)\leq 7/9$. For concrete examples of the pair
$(G,H)$, it is an interesting question to determine what the best
upper bound for $C(X)$ in this statement is. Denote this upper bound
by $C(G,H)$. We showed that $C(\Z,b\Z)=1 - \lfloor b^2/4\rfloor/b^2$; in particular
$C(\Z,b\Z)=3/4+o(1)$ as $b\rightarrow\infty$. Consider the
two-dimensional question of determining $C(\Z\times\Z,b\Z\times
b\Z)$. Clearly $C(\Z\times\Z,b\Z\times b\Z)\leq C(\Z, b\Z)$, and
we conjecture that $C(\Z\times\Z,b\Z\times b\Z) = C(\Z, b\Z)^2 = (9/16+o(1))$;
this bound may be attained by by taking $X=\{-(b-1)/2,\cdots,(b-1)/2\}\times\{-(b-1)/2,\cdots,(b-1)/2\}$.
We are unable to prove this conjecture; however, in \cite{Shao}, Shao makes
partial progress obtaining $C(\Z\times\Z,b\Z\times b\Z)\leq
1-3\sqrt{3}/4\pi\approx 0.59$; note that $9/16=0.5625$ so that Shao's bound is not
too far from our conjecture.  As mentioned earlier, another open problem  is to
extend Theorem 6.1 to other groups.

To end this paper, we make a final remark on Theorem
\ref{thm:approx_hom}. It is natural to wonder what can be said about
an $\epsilon$-approximate homomorphism $f$ for a small positive
constant $\epsilon$ (say $\epsilon=0.01$). We have already seen
that, in general, $f$ need not resemble a genuine homomorphism. On
the other hand, what one can conclude is that $f$ resembles a
genuine {\em local} homomorphism. In the special case when $G$ and
$H$ are vector spaces over finite fields, it turns out that any
$\epsilon$-approximate homomorphism {\em does} resemble a genuine
(global) homomorphism \cite{samorodnitsky}. A quantitative version
of this statement is equivalent to the polynomial Freiman-Ruzsa
(PFR) conjecture, a famous open problem in additive combinatorics.
See \cite{pfr} for the precise statement of this conjecture in the
finite field setting.

\bibliographystyle{plain}
\bibliography{CarriesRef}{}

\begin{thebibliography}{10}

\bibitem{alon}
N.~Alon.
\newblock Minimizing the number of carries in addition.
\newblock {\em SIAM J. Discrete Math.}, 27(1):562--566, 2013.

\bibitem{poker}
B.~Arkin, F.~Hill, S.~Marks, M.~Schmid, T.~J. Walls, and G.~Mc-Graw.
\newblock How we learned to cheat at online poker: {A} study in software
  security.
\newblock {\em The developer. com Journal}, 1999.

\bibitem{approx_hom}
M.~Ben-Or, D.~Coppersmith, M.~Luby, and R.~Rubinfeld.
\newblock Non-abelian homomorphism testing, and distributions close to their
  self-convolutions.
\newblock {\em Random Structures Algorithms}, 32(1):49--70, 2008.

\bibitem{bilu-lev-ruzsa}
Y.~F. Bilu, V.~F. Lev, and I.~Z. Ruzsa.
\newblock Rectification principles in additive number theory.
\newblock {\em Discrete Comput. Geom.}, 19(3, Special Issue):343--353, 1998.
\newblock Dedicated to the memory of Paul Erd{\H{o}}s.

\bibitem{blr}
M.~Blum, M.~Luby, and R.~Rubinfeld.
\newblock Self-testing/correcting with applications to numerical problems.
\newblock In {\em Proceedings of the 22nd {A}nnual {ACM} {S}ymposium on
  {T}heory of {C}omputing ({B}altimore, {MD}, 1990)}, volume~47, pages
  549--595, 1993.

\bibitem{borodin}
A.~Borodin, P.~Diaconis, and J.~Fulman.
\newblock On adding a list of numbers (and other one-dependent determinantal
  processes).
\newblock {\em Bull. Amer. Math. Soc. (N.S.)}, 47(4):639--670, 2010.

\bibitem{cajori}
F.~Cajori.
\newblock {\em A history of mathematical notations}, volume~1.
\newblock Dover Publications, 1993.

\bibitem{cauchy}
A.~Cauchy.
\newblock Sur les moyens d¡Ç{\'e}viter les erreurs dans les calculs
  num{\'e}riques.
\newblock {\em Comptes Rendus de l¡ÇAcad{\'e}mie des Sciences, Paris},
  11:789--798, 1840.

\bibitem{colson}
J.~Colson.
\newblock A {S}hort {A}ccount of {N}egativo-{A}ffirmative {A}rithmetick, by
  {M}r. {J}ohn {C}olson, {FRS}.
\newblock {\em Philosophical transactions}, 34(392-398):161--173, 1726.

\bibitem{diaconis2009a}
P.~Diaconis and J.~Fulman.
\newblock Carries, shuffling, and an amazing matrix.
\newblock {\em Amer. Math. Monthly}, 116(9):788--803, 2009.

\bibitem{DF}
D.~S. Dummit and R.~M. Foote.
\newblock {\em Abstract algebra}.
\newblock John Wiley \& Sons Inc., Hoboken, NJ, third edition, 2004.

\bibitem{fournier}
J.~Fournier.
\newblock Sharpness in {Y}oung's inequality for convolution.
\newblock {\em Pacific J. Math.}, 72(2):383--397, 1977.

\bibitem{gowers1}
W.~T. Gowers.
\newblock A new proof of {S}zemer\'edi's theorem for arithmetic progressions of
  length four.
\newblock {\em Geom. Funct. Anal.}, 8(3):529--551, 1998.

\bibitem{gowers2}
W.~T. Gowers.
\newblock A new proof of {S}zemer\'edi's theorem.
\newblock {\em Geom. Funct. Anal.}, 11(3):465--588, 2001.

\bibitem{Green}
B.~Green.
\newblock Additive combinatorics. {L}ecture notes for the {$22$}nd {M}c{G}ill
  invitational workshop on computational complexity.

\bibitem{pfr}
B.~Green.
\newblock Finite field models in additive combinatorics.
\newblock In {\em Surveys in combinatorics 2005}, volume 327 of {\em London
  Math. Soc. Lecture Note Ser.}, pages 1--27. Cambridge Univ. Press, Cambridge,
  2005.

\bibitem{green-ruzsa}
B.~Green and I.~Z. Ruzsa.
\newblock Sets with small sumset and rectification.
\newblock {\em Bull. London Math. Soc.}, 38(1):43--52, 2006.

\bibitem{green-tao1}
B.~Green and T.~Tao.
\newblock The primes contain arbitrarily long arithmetic progressions.
\newblock {\em Ann. of Math. (2)}, 167(2):481--547, 2008.

\bibitem{green-tao2}
B.~Green and T.~Tao.
\newblock Linear equations in primes.
\newblock {\em Ann. of Math. (2)}, 171(3):1753--1850, 2010.

\bibitem{holte}
J.~M. Holte.
\newblock Carries, combinatorics, and an amazing matrix.
\newblock {\em Amer. Math. Monthly}, 104(2):138--149, 1997.

\bibitem{isaksen}
D.~C. Isaksen.
\newblock A cohomological viewpoint on elementary school arithmetic.
\newblock {\em Amer. Math. Monthly}, 109(9):796--805, 2002.

\bibitem{Kemp}
J.~H.~B. Kemperman.
\newblock On small sumsets in an abelian group.
\newblock {\em Acta Math.}, 103:63--88, 1960.

\bibitem{knuth}
D.~E. Knuth.
\newblock {\em The art of computer programming}.
\newblock Addison-Wesley, 2006.

\bibitem{szemeredi-survey}
J.~Koml{\'o}s and M.~Simonovits.
\newblock Szemer\'edi's regularity lemma and its applications in graph theory.
\newblock In {\em Combinatorics, {P}aul {E}rd{\H o}s is eighty, {V}ol.\ 2
  ({K}eszthely, 1993)}, volume~2 of {\em Bolyai Soc. Math. Stud.}, pages
  295--352. J\'anos Bolyai Math. Soc., Budapest, 1996.

\bibitem{laba}
I.~{\L}aba.
\newblock From harmonic analysis to arithmetic combinatorics.
\newblock {\em Bull. Amer. Math. Soc. (N.S.)}, 45(1):77--115, 2008.

\bibitem{lev}
V.~F. Lev.
\newblock Linear equations over {$\Bbb F_p$} and moments of exponential sums.
\newblock {\em Duke Math. J.}, 107(2):239--263, 2001.

\bibitem{Lev2}
V.~F. Lev.
\newblock Distribution of points on arcs.
\newblock {\em Integers}, 5(2):A11, 6, 2005.

\bibitem{nathanson}
M.~B. Nathanson.
\newblock {\em Additive number theory: {I}nverse problems and the geometry of
  sumsets}, volume 165 of {\em Graduate Texts in Mathematics}.
\newblock Springer-Verlag, New York, 1996.

\bibitem{pollard}
J.~M. Pollard.
\newblock Addition properties of residue classes.
\newblock {\em J. London Math. Soc. (2)}, 11(2):147--152, 1975.

\bibitem{samorodnitsky}
A.~Samorodnitsky.
\newblock Low-degree tests at large distances.
\newblock In {\em S{TOC}'07---{P}roceedings of the 39th {A}nnual {ACM}
  {S}ymposium on {T}heory of {C}omputing}, pages 506--515. ACM, New York, 2007.

\bibitem{Ser}
O.~Serra.
\newblock An isoperimetric method for the small sumset problem.
\newblock In {\em Surveys in combinatorics 2005}, volume 327 of {\em London
  Math. Soc. Lecture Note Ser.}, pages 119--152. Cambridge Univ. Press,
  Cambridge, 2005.

\bibitem{Shao}
X.~Shao.
\newblock Large values of the additive energy in {${\Bbb R}^d$} and {${\Bbb
  Z}^d$}.
\newblock {\em Preprint}.

\bibitem{szemeredi}
E.~Szemer{\'e}di.
\newblock Regular partitions of graphs.
\newblock In {\em Probl\`emes combinatoires et th\'eorie des graphes ({C}olloq.
  {I}nternat. {CNRS}, {U}niv. {O}rsay, {O}rsay, 1976)}, volume 260 of {\em
  Colloq. Internat. CNRS}, pages 399--401. CNRS, Paris, 1978.

\bibitem{higher-fourier}
T.~Tao.
\newblock {\em Higher order {F}ourier analysis}, volume 142 of {\em Graduate
  Studies in Mathematics}.
\newblock American Mathematical Society, Providence, RI, 2012.

\bibitem{trevisan}
L.~Trevisan.
\newblock Guest column: additive combinatorics and theoretical computer
  science.
\newblock {\em ACM SIGACT News}, 40(2):50--66, 2009.

\end{thebibliography}

\end{document}